\documentclass[10pt]{amsart}
\usepackage{amsmath,amssymb,amsthm,mathtools,mathrsfs}
\usepackage{microtype}
\usepackage[hidelinks]{hyperref}
\usepackage[nameinlink,noabbrev]{cleveref}

\allowdisplaybreaks
\newtheorem{theorem}{Theorem}[section]
\newtheorem{corollary}[theorem]{Corollary}
\newtheorem{proposition}[theorem]{Proposition}
\newtheorem{lemma}[theorem]{Lemma}
\newtheorem{remark}[theorem]{Remark}
\theoremstyle{definition}

\newcommand{\R}{\mathbb R}
\newcommand{\e}{\mathrm e}
\newcommand{\dd}{\,\mathrm d}

\newcommand{\LtLp}[2]{L^{#1}(0,T;L^{#2}(D))}
\newcommand{\Gv}{\Gamma_{\!V}}
\newcommand{\Gh}{\Gamma_{\!H}}

\title[Pure-swirl loss of boundedness and mixed norms]
{Pure-swirl loss of boundedness under $L^1_tL^2_x$ forcing:
exact mixed-norm ranges}

\author{Hugo Beir\~ao da Veiga}
\address{Department of Mathematics, University of Pisa, Pisa, Italy}
\address{Academia das Ci\^encias de Lisboa, Lisbon, Portugal}
\email{hbeiraodaveiga@gmail.com}

\author{Jiaqi Yang}
\address{School of Mathematics and Statistics, Northwestern Polytechnical University,
Xi'an, China}
\email{yjqmath@nwpu.edu.cn; yjqmath@163.com}


\begin{document}

\begin{abstract}
We construct an explicit pure-swirl solution of the forced three-dimensional
Navier--Stokes equations in a circular cylinder.  For every $T>0$ and
$1\le p,q<\infty$ with $1/p+1/q>1$, the force belongs to
$L^q(0,T;L^p(D))$ and is smooth for $t<T$.  Moreover, the same construction
always yields the additional energy-class property
$f\in L^1(0,T;L^2(D))$.  The solution is classical on $[0,T)$, extends
strongly in $L^2(D)$ to the unique
Leray--Hopf solution at time $T$, and satisfies the energy equality, while
$\|v(t)\|_{L^\infty(D)}\to\infty$ as $t\uparrow T$.  We determine the
exact mixed-norm ranges of both the force and the velocity and obtain two-sided
rates for the force, the velocity supremum norm and the enstrophy.  The
construction refines Zhang's profile by an annular cancellation that preserves
smoothness at the symmetry axis.  It is a direct, self-contained sharpening of
the $k=1$ part of our previous weighted construction and supersedes its endpoint
discussion.  Since the convection term is absorbed by the pressure, the same
example applies to the forced Stokes system.
\end{abstract}

\maketitle

\section{Introduction and main results}

We study the forced incompressible Navier--Stokes
system
\begin{equation}
\label{eq:NS}
\left\{
\begin{aligned}
    \partial_t v-\Delta v+(v\cdot\nabla)v+\nabla P&=f
       &&\text{in }D\times(0,T),\\
    \nabla\cdot v&=0
       &&\text{in }D\times(0,T),
\end{aligned}
\right.
\end{equation}
where
\[
    D=B_2(0,1)\times(0,1)\subset\R^3
\]
is the unit circular cylinder. 
Write
\[
    \Gv=\partial B_2(0,1)\times(0,1),
    \qquad
    \Gh=B_2(0,1)\times\{0,1\}.
\]
On the vertical boundary we impose no slip, and on the horizontal boundary we
impose the free-slip (zero tangential stress) condition
\begin{equation}
\label{eq:boundary}
    v=0\quad\text{on }\Gv,
    \qquad
    v\cdot n=0,
    \quad
    \bigl((\nabla v+\nabla v^{\mathsf T})n\bigr)_{\tau}=0
    \quad\text{on }\Gh.
\end{equation}
For the pure-swirl fields constructed below, the condition on $\Gh$ reduces simply
to
\[
    v_3=0,
    \qquad
    \partial_{x_3}v_1=\partial_{x_3}v_2=0.
\]

Leray--Hopf solutions exist globally in the energy class
\cite{Leray1934,Hopf1951,Ladyzhenskaya1969,Sohr2001}; conditional regularity
and weak--strong uniqueness follow from the Serrin theory and its endpoint
refinements \cite{Serrin1963,EscauriazaSereginSverak2003,FarwigSohrVarnhorn2012}.
Forced weak solutions may nevertheless display nonuniqueness or loss of
boundedness; see \cite{ABC2022,GaldiGazzola2026}.  These phenomena are distinct
from the unforced smooth-data regularity problem.

A direct forced blow-up construction in a finite cylinder was recently given by
Zhang \cite{Zhang2025}.  His velocity is purely azimuthal and the system reduces
to a scalar parabolic equation; he obtains a force in $L^q_tL^1_x$ for every
finite $q>1$.  The self-similar scalar profile is inspired by Leray's classical
ansatz, although nontrivial unforced backward self-similar solutions are ruled out
under natural integrability assumptions by Ne\v{c}as, R\r{u}\v{z}i\v{c}ka and
\v{S}ver\'ak \cite{NecasRuzickaSverak1996}.  Thus the singular forcing is an
essential part of Zhang's mechanism.

\begin{remark}[The pure-swirl reduction]
\label{rem:old-remark}
We recall here, in the notation of the present paper, the observation made in
\cite[Remark~2.1]{BDVY2024}.  It is worth emphasizing that the assumption $v_r=v_3=0$ makes the nonlinear
term dynamically inessential for the azimuthal equation.  More precisely, for
an axisymmetric pure-swirl field $v=v_\theta(r,t)e_\theta$,
\[
    (v\cdot\nabla)v=-\frac{v_\theta^2}{r}e_r.
\]
Thus the convective term has no azimuthal component and its radial component is
absorbed exactly into the pressure.  Consequently, the construction and its
conclusions also apply to the Stokes evolution problem.  In particular, within
this ansatz, the present loss-of-boundedness mechanism cannot occur under a
bounded external force.  This observation does not make the construction less
significant: the argument isolates the effect of the forcing without any
nonlinear amplification.

Furthermore, the result is obtained under minimal kinematic assumptions: the
velocity is purely azimuthal and its scalar amplitude depends on a single
spatial variable.  This suggests that allowing an additional velocity component
or a genuine dependence on another spatial variable, even in a minimally
invasive way, may lead to richer phenomena.  Finally, both the force and the
solution are smooth on $D\times[0,T)$, and the loss of boundedness occurs only
as $t\uparrow T$.
\end{remark}

Concerning related results on axisymmetric Navier--Stokes flows, we recall the
lower blow-up-rate estimates of Chen, Strain, Tsai and Yau
\cite{ChenStrainTsaiYau2008,ChenStrainTsaiYau2009}, the regularity criteria of
Chen, Fang and Zhang \cite{ChenFangZhang2017}, the Liouville and singularity
analysis of Koch, Nadirashvili, Seregin and \v{S}ver\'ak
\cite{KochNadirashviliSereginSverak2009}, and the slightly supercritical
regularity result of Pan \cite{Pan2016}.  We also mention the work of Li, Pan,
Yang, Zeng, Zhang and Zhao \cite{LiPanYangZengZhangZhao2024} on axisymmetric
flows passing a cone; as a by-product, they constructed an unbounded
finite-energy solution of a forced Navier--Stokes problem in a special
cusp-type domain.

The present paper is a continuation and sharpening of \cite{BDVY2024}, rather
than an independent variant of Zhang's construction.  In \cite{BDVY2024} we
introduced the weighted ansatz $r^\alpha\phi$ and treated the more general
hierarchy $k\geq1$, obtaining loss of boundedness of derivatives of order
$k-1$.  The case $k=1$ was then extracted from that general argument.  The
sufficient mixed-norm condition stated there, however, did not retain the
$p$-th root when passing from an estimate for $\|f(t)\|_{L^p}^p$ to the time
norm.  Keeping this factor gives the correct condition
\[
    ((3-\alpha)p-2)q<2p,
\]
which includes $(p,q)=(2,1)$ and hence reaches $L^1_tL^2_x$.

The purpose of the present note is to isolate this $k=1$ case and give a direct,
reasonably self-contained proof.  Besides the corrected endpoint condition, the
new ingredients are the annular cancellation that guarantees smoothness at the
symmetry axis for noninteger $\alpha$, the two-sided temporal asymptotics, the
exact mixed-norm classifications of both $f$ and $v$, and the proof of a strong
$L^2$ terminal trace, the energy equality and uniqueness in the Leray--Hopf
class.  Thus the present paper supersedes the corresponding $k=1$ endpoint
discussion in \cite{BDVY2024}; the higher-order results for $k\geq2$ obtained
there are not reconsidered here.  To keep the argument self-contained, we
repeat the pure-swirl reduction and the profile calculations needed below.

We now state the main result; the mixed-boundary Leray--Hopf class is recalled below.

\begin{theorem}
\label{thm:main}
Let $T>0$ and let $1\leq p,q<\infty$ satisfy
\begin{equation}
\label{eq:pq-condition}
    \frac1p+\frac1q>1.
\end{equation}
Then there exist a force
\begin{equation}
\label{eq:force-spaces}
    f\in \LtLp{q}{p}
\end{equation}
and a classical solution $(v,P)$ of \eqref{eq:NS}--\eqref{eq:boundary} on
$D\times[0,T)$ with smooth initial datum such that
\begin{equation}
\label{eq:energy-class}
    v\in L^\infty(0,T;L^2(D))\cap L^2(0,T;H^1(D))
\end{equation}
and
\begin{equation}
\label{eq:Linfty-blowup}
    \lim_{t\uparrow T}\|v(t)\|_{L^\infty(D)}=\infty.
\end{equation}
In addition, the force constructed above always satisfies
\begin{equation}
\label{eq:energy-force-extra}
    f\in L^1(0,T;L^2(D)).
\end{equation}
Consequently, the solution extends to a function $v\in C([0,T];L^2(D))$.
This extension is the unique Leray--Hopf solution on $[0,T]$ with the same
force, initial datum and boundary conditions, and it satisfies the energy
equality on $[0,T]$.

More precisely, one may choose a parameter
\begin{equation}
\label{eq:alpha-intro}
    \max\left\{0,3-\frac2p-\frac2q\right\}<\alpha<1
\end{equation}
such that, as $t\uparrow T$,
\begin{align}
\label{eq:rates-intro}
    \|f(t)\|_{L^p(D)}
       &\asymp (T-t)^{-\frac{3-\alpha}{2}+\frac1p},\\
    \|v(t)\|_{L^\infty(D)}
       &\asymp (T-t)^{-\frac{1-\alpha}{2}},\\
    \|\nabla v(t)\|_{L^2(D)}^2
       &\asymp (T-t)^{\alpha-1}.
\end{align}
The velocity itself has the following exact mixed-norm classification.  For
$1\leq m,\sigma<\infty$,
\begin{equation}
\label{eq:velocity-mixed-intro}
    v\in L^\sigma(0,T;L^m(D))
    \quad\Longleftrightarrow\quad
    \alpha>1-\frac2m-\frac2\sigma.
\end{equation}
Moreover,
\begin{equation}
\label{eq:velocity-endpoints-intro}
\begin{aligned}
    v\in L^\infty(0,T;L^m(D))
       &\quad\Longleftrightarrow\quad \alpha>1-\frac2m,
       &&1\leq m<\infty,\\
    v\in L^\sigma(0,T;L^\infty(D))
       &\quad\Longleftrightarrow\quad \alpha>1-\frac2\sigma,
       &&1\leq\sigma<\infty.
\end{aligned}
\end{equation}
\end{theorem}

\begin{corollary}
\label{cor:curve}
For every finite $p\ge1$, one may take $q=1$; in particular the construction
reaches $f\in L^1(0,T;L^2(D))$.  Moreover, if $1\le p\le2$ and
\[
    q=\frac{4p}{7p-6},\qquad \frac2q+\frac3p=\frac72,
\]
then the conclusion of \cref{thm:main} holds, including $(p,q)=(2,1)$.
For any prescribed finite $m,\sigma$, one may also arrange
$v\in L^\sigma_tL^m_x$ while $\|v(t)\|_\infty\to\infty$.
\end{corollary}

\begin{remark}
\label{rem:context}
The region $1/p+1/q>1$ lies
beyond the parabolic boundedness threshold $2/q+3/p=2$, distinct from the force
scaling line $2/q+3/p=3$.  The solution is classical on every $[0,T')$ and
loses boundedness only at $T$, where the force becomes unbounded.  Since the
pressure cancels convection, the example also solves Stokes and does not address
the Millennium Prize problem.
\end{remark}

\section{Leray--Hopf setting and pure-swirl reduction}
\label{sec:LH}
\label{sec:reduction}
Let $\mathscr V$ be the smooth divergence-free fields satisfying
$\varphi=0$ on $\Gv$ and $\varphi\cdot n=0$ on $\Gh$, and let $H$ and $V$
be its closures in $L^2(D)$ and $H^1(D)$.  A Leray--Hopf solution belongs to
\[
 L^\infty(0,T;H)\cap L^2(0,T;V)\cap C_{\rm w}([0,T];H)
\]
and satisfies the standard weak formulation and energy inequality; see
\cite{Fursikov1982,Galdi2000,Sohr2001}.  On the flat boundary $\Gh$ the
natural condition is $u\cdot n=0$ and $\partial_{x_3}u_\tau=0$.
For every admissible test field $\varphi$, the weak formulation is
\begin{align}
\label{eq:weak-form}
 &-\int_0^T\!\int_D v\cdot\partial_t\varphi
 +\int_0^T\!\int_D\nabla v:\nabla\varphi
 -\int_0^T\!\int_D(v\otimes v):\nabla\varphi\notag\\
 &\hspace{25mm}=\int_Dv_0\cdot\varphi(0)
 +\int_0^T\!\int_Df\cdot\varphi .
\end{align}

Use cylindrical coordinates $(r,\theta,x_3)$ and write
\[
    e_r=(\cos\theta,\sin\theta,0),
    \qquad
    e_\theta=(-\sin\theta,\cos\theta,0).
\]
Consider
\begin{equation}
\label{eq:pure-swirl}
    v(x,t)=W(r,t)e_\theta,
    \qquad
    f(x,t)=F(r,t)e_\theta.
\end{equation}
Then $\nabla\cdot v=0$ and
\[
    (v\cdot\nabla)v=-\frac{W^2}{r}e_r.
\]
If
\begin{equation}
\label{eq:pressure}
    P(r,t)=\int_0^r\frac{W(\ell,t)^2}{\ell}\,\dd\ell,
\end{equation}
then the radial equation is satisfied, whereas the azimuthal equation becomes
\begin{equation}
\label{eq:scalar-standard}
    \partial_tW-
    \left(\partial_r^2+\frac1r\partial_r-\frac1{r^2}\right)W=F.
\end{equation}
A field of the form \eqref{eq:pure-swirl} is independent of $x_3$, has zero normal
component on $\Gh$, and satisfies the free-slip condition there.  Thus only the
condition
\[
    W(1,t)=0
\]
remains to be imposed.

With \eqref{eq:pressure},
\[
    (v\cdot\nabla)v+\nabla P=0,
\]
so every example below also solves the corresponding forced Stokes system.

\section{Annular profile and construction}
\label{sec:profile}

Fix the prescribed $T>0$.  Choose
\begin{equation}
\label{eq:ab-choice}
    0<a<b<\min\left\{1,\frac1{\sqrt{2T}}\right\}
\end{equation}
and a nonnegative, nonzero function
\[
    A\in C_c^\infty((a,b)).
\]
Set
\begin{equation}
\label{eq:k-def}
    k(\rho)=-\e^{\rho^2/2}A'(\rho)
\end{equation}
and
\begin{equation}
\label{eq:phi0-def}
    \phi_0(\rho)
    =-\frac1\rho\int_0^\rho s\e^{s^2/2}A(s)\,\dd s,
    \qquad \rho>0,
\end{equation}
with $\phi_0(0)=0$.  Finally, define
\begin{equation}
\label{eq:beta-def}
    \beta=\int_0^\infty s\e^{s^2/2}A(s)\,\dd s>0.
\end{equation}
Since
\[
    A(\rho)=\int_\rho^\infty\e^{-\ell^2/2}k(\ell)\,\dd\ell,
    \qquad
    \int_0^\infty\e^{-\ell^2/2}k(\ell)\,\dd\ell=0,
\]
formula \eqref{eq:phi0-def} is the profile formula used in Zhang's construction.
The displayed cancellation, generated by the annular choice of $A$, forces the
profile to vanish near the symmetry axis and is the key to smoothness for
noninteger $\alpha$.

\begin{lemma}
\label{lem:profile}
The function $\phi_0$ belongs to $C^\infty([0,\infty))$ and satisfies
\begin{equation}
\label{eq:profile-ode}
    \phi_0''+\frac1\rho\phi_0'-\frac1{\rho^2}\phi_0
    -\phi_0-\rho\phi_0'=k(\rho).
\end{equation}
Moreover,
\begin{equation}
\label{eq:inner-outer-profile}
    \phi_0(\rho)=0\quad(0\leq\rho\leq a),
    \qquad
    \phi_0(\rho)=-\frac\beta\rho\quad(\rho\geq b),
\end{equation}
and
\begin{equation}
\label{eq:profile-bounds}
    |\phi_0(\rho)|\leq C\frac{\rho}{1+\rho^2},
    \qquad
    |\phi_0'(\rho)|\leq\frac{C}{1+\rho^2}.
\end{equation}
\end{lemma}

\begin{proof}
Let $g(\rho)=\rho\phi_0(\rho)$.  From \eqref{eq:phi0-def},
\[
    g'(\rho)=-\rho\e^{\rho^2/2}A(\rho).
\]
Using \eqref{eq:k-def}, direct differentiation gives
\[
    g''-\left(\frac1\rho+\rho\right)g'=\rho k(\rho).
\]
Substituting $g=\rho\phi_0$ and dividing by $\rho$ gives
\eqref{eq:profile-ode}.  Since $A$ is supported in $(a,b)$, the integral in
\eqref{eq:phi0-def} vanishes for $\rho\leq a$ and equals $\beta$ for
$\rho\geq b$.  This proves \eqref{eq:inner-outer-profile} and smoothness across
the joining points.  The bounds follow from the exact inner and outer formulas and
smoothness on $[a,b]$.
\end{proof}

For $0\leq t<T$, write
\[
    R(t)=\sqrt{2(T-t)}
\]
and define
\begin{equation}
\label{eq:phi-h-def}
    \phi(r,t)=\frac1{R(t)}\phi_0\left(\frac r{R(t)}\right),
    \qquad
    h(r,t)=\frac1{R(t)^3}k\left(\frac r{R(t)}\right).
\end{equation}
A direct computation from \eqref{eq:profile-ode} gives
\begin{equation}
\label{eq:phi-equation}
    \left(\partial_r^2+\frac1r\partial_r-\frac1{r^2}\right)\phi
    -\partial_t\phi=h.
\end{equation}
The profile bounds imply
\begin{equation}
\label{eq:phi-pointwise}
    |\phi(r,t)|\leq C\frac r{r^2+T-t},
    \qquad
    |\partial_r\phi(r,t)|\leq\frac C{r^2+T-t}.
\end{equation}
Furthermore,
\begin{equation}
\label{eq:inner-zero}
    \phi(r,t)=h(r,t)=0
    \quad\text{for }0\leq r\leq aR(t),
\end{equation}
and
\begin{equation}
\label{eq:outer-phi}
    \phi(r,t)=-\frac\beta r,
    \quad
    \partial_r\phi(r,t)=\frac\beta{r^2}
    \quad\text{for }r\geq bR(t).
\end{equation}

\section{Construction and estimates}
\label{sec:construction}

Fix $0<\alpha<1$ and set
\begin{equation}
\label{eq:W-def}
    W(r,t)=r^\alpha\phi(r,t)+\beta r.
\end{equation}
Define $v=We_\theta$, define $P$ by \eqref{eq:pressure}, and set
\begin{equation}
\label{eq:F-def}
    F=-r^\alpha h
      -\alpha^2r^{\alpha-2}\phi
      -2\alpha r^{\alpha-1}\partial_r\phi,
    \qquad
    f=Fe_\theta.
\end{equation}
Indeed, the product rule and \eqref{eq:phi-equation} yield
\[
    \partial_tW-
    \left(\partial_r^2+\frac1r\partial_r-\frac1{r^2}\right)W=F,
\]
so \cref{sec:reduction} gives a solution of \eqref{eq:NS}.

For each $t<T$, \eqref{eq:inner-zero} gives $W=\beta r$ and $f=0$ near the
axis, so the noninteger factor $r^\alpha$ causes no loss of smoothness.
Moreover $x_3$-independence gives the condition on $\Gh$, while
$1/R(t)>b$ and \eqref{eq:outer-phi} imply $W(1,t)=0$.  Thus $v,P,f$ are smooth
on $\overline D\times[0,T)$ and satisfy \eqref{eq:boundary}.

\begin{lemma}[Exact force rate]
\label{lem:force-rate}
For every $1\leq m<\infty$ there exist constants $c_m,C_m>0$ and
$t_m<T$ such that
\begin{equation}
\label{eq:force-rate-m}
    c_m(T-t)^{-\frac{3-\alpha}{2}+\frac1m}
    \leq \|f(t)\|_{L^m(D)}
    \leq C_m(T-t)^{-\frac{3-\alpha}{2}+\frac1m}
\end{equation}
for $t_m<t<T$.  Consequently, for $1\leq \sigma,m<\infty$,
\begin{equation}
\label{eq:force-iff}
    f\in L^\sigma(0,T;L^m(D))
    \quad\Longleftrightarrow\quad
    \alpha>3-\frac2m-\frac2\sigma.
\end{equation}
In particular, $f\in L^1(0,T;L^2(D))$ for every $\alpha>0$.
\end{lemma}

\begin{proof}
Set $s=T-t$ and
\[
    B_m=\frac{(3-\alpha)m-2}{2}>0.
\]
Because $k$ is supported in $(a,b)$, cylindrical coordinates give
\begin{align}
\label{eq:h-upper}
    \int_D|r^\alpha h|^m\,\dd x
    &\leq C R(t)^{-3m}
       \int_0^{bR(t)}r^{\alpha m+1}\,\dd r
      \leq Cs^{-B_m}.
\end{align}
By \eqref{eq:phi-pointwise},
\[
    \left|
      \alpha^2r^{\alpha-2}\phi
      +2\alpha r^{\alpha-1}\partial_r\phi
    \right|
    \leq C\frac{r^{\alpha-1}}{r^2+s}.
\]
The left-hand side vanishes for $r\leq aR(t)$.  Hence, after
$r=\sqrt{s}\,y$,
\begin{align}
\label{eq:remainder-upper}
\int_D
       \left|
         \alpha^2r^{\alpha-2}\phi
         +2\alpha r^{\alpha-1}\partial_r\phi
       \right|^m\,\dd x\leq
       Cs^{-B_m}
       \int_{a\sqrt2}^{\infty}
       \frac{y^{1+\alpha m-m}}{(1+y^2)^m}\,\dd y
       \leq Cs^{-B_m}.
\end{align}
The final integral is finite because $m(3-\alpha)>2$, which follows from
$m\geq1$ and $\alpha<1$.  This proves the upper bound in
\eqref{eq:force-rate-m}.

For the lower bound, fix $r_0\in(0,1)$.  When $bR(t)\leq r\leq r_0$,
\eqref{eq:outer-phi} and \eqref{eq:F-def} give the exact identity
\begin{equation}
\label{eq:F-outer}
    F(r,t)=-\beta\alpha(2-\alpha)r^{\alpha-3}.
\end{equation}
Therefore
\[
    \|f(t)\|_{L^m(D)}^m
    \geq c\int_{bR(t)}^{r_0}r^{1-m(3-\alpha)}\,\dd r
    \geq cs^{-B_m}
\]
for $t$ sufficiently close to $T$.  This proves \eqref{eq:force-rate-m}.
The equivalence \eqref{eq:force-iff} follows by integrating
$s^{-\sigma((3-\alpha)/2-1/m)}$ near $s=0$.  At equality the integral diverges
logarithmically.  Taking $(\sigma,m)=(1,2)$ gives the last assertion.
\end{proof}

\begin{lemma}[Velocity and energy rates]
\label{lem:velocity-rates}
There exist $c,C>0$ and $t_0<T$ such that, for $t_0<t<T$,
\begin{equation}
\label{eq:velocity-rate}
    c(T-t)^{-\frac{1-\alpha}{2}}
    \leq\|v(t)\|_{L^\infty(D)}
    \leq C(T-t)^{-\frac{1-\alpha}{2}}
\end{equation}
and
\begin{equation}
\label{eq:gradient-rate}
    c(T-t)^{\alpha-1}
    \leq\|\nabla v(t)\|_{L^2(D)}^2
    \leq C(T-t)^{\alpha-1}.
\end{equation}
Moreover,
\begin{equation}
\label{eq:L2-uniform}
    \sup_{0\leq t<T}\|v(t)\|_{L^2(D)}<\infty.
\end{equation}
\end{lemma}

\begin{proof}
From \eqref{eq:phi-pointwise},
\[
    |W(r,t)|\leq C\frac{r^{\alpha+1}}{r^2+s}+Cr,
    \qquad s=T-t.
\]
Writing $r=\sqrt{s}\,y$ shows that
\[
    \sup_{r\in[0,1]}\frac{r^{\alpha+1}}{r^2+s}
    \leq Cs^{(\alpha-1)/2},
\]
because $y^{\alpha+1}/(1+y^2)$ is bounded for $0<\alpha<1$.  This proves the
upper bound in \eqref{eq:velocity-rate}.  For any fixed $\rho_0>b$, put
$r(t)=\rho_0R(t)$.  For $t$ close to $T$, $r(t)<1$, and
\eqref{eq:outer-phi} yields
\[
    W(r(t),t)
    =-\beta\rho_0^{\alpha-1}R(t)^{\alpha-1}
      +\beta\rho_0R(t).
\]
The first term dominates, proving the lower bound.

For a pure swirl independent of $x_3$,
\begin{equation}
\label{eq:grad-swirl}
    |\nabla v|^2=|\partial_rW|^2+\frac{|W|^2}{r^2}.
\end{equation}
The pointwise bounds on $\phi$ and $\partial_r\phi$ give
\[
    |\partial_rW|+\frac{|W|}{r}
    \leq C\frac{r^\alpha}{r^2+s}+C.
\]
Consequently,
\begin{align}
\|\nabla v(t)\|_2^2\leq C\int_0^1\frac{r^{2\alpha+1}}{(r^2+s)^2}\,\dd r+C\leq Cs^{\alpha-1}
      \int_0^\infty\frac{y^{2\alpha+1}}{(1+y^2)^2}\,\dd y+C
     \leq Cs^{\alpha-1}.
\end{align}
For the reverse inequality, in the outer region one has
\[
    W(r,t)=\beta(r-r^{\alpha-1}).
\]
Choosing a fixed small $r_0>0$, the term $|W|/r$ is bounded below by
$cr^{\alpha-2}$ on $bR(t)\leq r\leq r_0$ for all $t$ close enough to $T$.
Thus
\[
    \|\nabla v(t)\|_2^2
    \geq c\int_{bR(t)}^{r_0}r^{2\alpha-3}\,\dd r
    \geq cs^{\alpha-1},
\]
which proves \eqref{eq:gradient-rate}.

Finally,
\begin{align}
    \|v(t)\|_2^2
    &\leq C\int_0^1\frac{r^{2\alpha+3}}{(r^2+s)^2}\,\dd r+C\\
    &=Cs^\alpha\int_0^{1/\sqrt{s}}
       \frac{y^{2\alpha+3}}{(1+y^2)^2}\,\dd y+C
      \leq C,
\end{align}
where the last estimate follows by splitting the $y$-integral at $1$.
This proves \eqref{eq:L2-uniform}.
\end{proof}

\begin{lemma}[Exact mixed norms of the velocity]
\label{lem:velocity-mixed}
For every $1\leq m<\infty$, as $t\uparrow T$,
\begin{equation}
\label{eq:velocity-Lm-rate}
\|v(t)\|_{L^m(D)}
\asymp
\begin{cases}
1,
   &m(1-\alpha)<2,\\[1mm]
|\log(T-t)|^{1/m},
   &m(1-\alpha)=2,\\[1mm]
(T-t)^{-\frac{1-\alpha}{2}+\frac1m},
   &m(1-\alpha)>2.
\end{cases}
\end{equation}
Consequently, for $1\leq m,\sigma<\infty$,
\begin{equation}
\label{eq:velocity-mixed-iff}
    v\in L^\sigma(0,T;L^m(D))
    \quad\Longleftrightarrow\quad
    \alpha>1-\frac2m-\frac2\sigma.
\end{equation}
Furthermore,
\begin{equation}
\label{eq:velocity-time-endpoint}
    v\in L^\infty(0,T;L^m(D))
    \quad\Longleftrightarrow\quad
    \alpha>1-\frac2m,
    \qquad 1\leq m<\infty,
\end{equation}
and, for $1\leq\sigma<\infty$,
\begin{equation}
\label{eq:velocity-space-endpoint}
    v\in L^\sigma(0,T;L^\infty(D))
    \quad\Longleftrightarrow\quad
    \alpha>1-\frac2\sigma.
\end{equation}
In particular, the terminal profile belongs to $L^m(D)$ precisely when
$\alpha>1-2/m$; in this case $v(t)\to v_T$ strongly in $L^m(D)$.
\end{lemma}

\begin{proof}
Put $s=T-t$.  The pointwise estimate used in the proof of
\cref{lem:velocity-rates} gives
\begin{equation}
\label{eq:velocity-Lm-upper-integral}
    \|v(t)\|_{L^m(D)}^m
    \leq C\int_0^1
       \frac{r^{m(\alpha+1)+1}}{(r^2+s)^m}\,\dd r+C.
\end{equation}
After the change of variables $r=\sqrt{s}\,y$, the integral on the right is
\[
    s^{1-\frac{m(1-\alpha)}2}
    \int_0^{1/\sqrt{s}}
       \frac{y^{m(\alpha+1)+1}}{(1+y^2)^m}\,\dd y.
\]
The large-$y$ behavior of the integrand is
$y^{1-m(1-\alpha)}$.  It follows that the right-hand side of
\eqref{eq:velocity-Lm-upper-integral} is bounded respectively by a constant,
by $C|\log s|$, or by $Cs^{1-m(1-\alpha)/2}$ according as
$m(1-\alpha)<2$, $m(1-\alpha)=2$, or $m(1-\alpha)>2$.

For the reverse estimates, choose $r_0\in(0,1)$ sufficiently small.  In the
outer region $bR(t)\leq r\leq r_0$ one has
\[
    W(r,t)=\beta(r-r^{\alpha-1}),
\]
and hence $|W(r,t)|\geq c r^{\alpha-1}$ for all $t$ sufficiently close to
$T$.  Therefore
\[
    \|v(t)\|_{L^m(D)}^m
    \geq c\int_{bR(t)}^{r_0}r^{1-m(1-\alpha)}\,\dd r.
\]
The three alternatives in \eqref{eq:velocity-Lm-rate} now follow by evaluating
this elementary integral.  Integrating the resulting power, or logarithm, in
time proves \eqref{eq:velocity-mixed-iff}.  The endpoint statements
\eqref{eq:velocity-time-endpoint} and \eqref{eq:velocity-space-endpoint} follow
from \eqref{eq:velocity-Lm-rate} and \eqref{eq:velocity-rate}, respectively.
Finally, $v_T(r)=\beta(r-r^{\alpha-1})e_\theta$ belongs to $L^m(D)$ exactly
when $m(1-\alpha)<2$.  In that case the domination
$|W(r,t)|\leq C(r^{\alpha-1}+r)$ yields strong $L^m$ convergence by dominated
convergence.
\end{proof}

The rate is realized on the scale $r\asymp R(t)$, whereas for fixed $r>0$
we have $W(r,t)\to W_T(r)=\beta(r-r^{\alpha-1})$; thus the terminal loss of
boundedness is concentrated on the symmetry axis.

\begin{proposition}[Final trace, energy equality and uniqueness]
\label{prop:extension}
The solution constructed above extends to $v\in C([0,T];L^2(D))$, belongs to
$L^2(0,T;V)$, satisfies \eqref{eq:weak-form} and the energy equality on
$[0,T]$, and is the unique Leray--Hopf solution with the same data.
\end{proposition}

\begin{proof}
For every fixed $r>0$, the outer formula applies for all $t$ sufficiently close to
$T$, so
\begin{equation}
\label{eq:terminal-profile}
    W(r,t)\longrightarrow W_T(r):=\beta(r-r^{\alpha-1}).
\end{equation}
Moreover, \eqref{eq:phi-pointwise} implies
\[
    |W(r,t)|\leq C(r^{\alpha-1}+r).
\]
The right-hand side belongs to $L^2(D)$ precisely because $\alpha>0$.
Dominated convergence gives
\begin{equation}
\label{eq:strong-trace}
    v(t)\longrightarrow v_T:=W_Te_\theta
    \quad\text{strongly in }L^2(D).
\end{equation}
Set $v(T)=v_T$.  Smoothness before $T$ and \eqref{eq:strong-trace} imply
$v\in C([0,T];L^2(D))$.  By \eqref{eq:gradient-rate} and $\alpha>0$,
$v\in L^2(0,T;V)$.

For every $t<T$, multiplying \eqref{eq:NS} by $v$ and integrating gives
\begin{equation}
\label{eq:energy-equality}
    \frac12\|v(t)\|_2^2
    +\int_0^t\|\nabla v(s)\|_2^2\,\dd s
    =\frac12\|v(0)\|_2^2
     +\int_0^t\!\int_Df\cdot v\,\dd x\dd s.
\end{equation}
The pressure and transport terms vanish by incompressibility and the normal
boundary conditions.  The viscous boundary term vanishes because $v=0$ on
$\Gv$, while $\partial_{x_3}v=0$ on $\Gh$.  Since
$f\in L^1(0,T;L^2(D))$, $v\in L^\infty(0,T;L^2(D))$, and
$\nabla v\in L^2(0,T;L^2(D))$, all terms in \eqref{eq:energy-equality} converge
as $t\uparrow T$.  Thus the equality also holds at $T$.

Passing to the limit in the classical weak formulation on $[0,t]$ as
$t\uparrow T$ gives \eqref{eq:weak-form}; the nonlinear term is integrable in
the standard energy class.  Hence the extension is a Leray--Hopf solution.

Let $u$ be another Leray--Hopf solution with the same data.  Fix $T'<T$ and set
$w=u-v$.  Since $v$ is smooth on $[0,T']$, the usual weak--strong difference
estimate (see, e.g., \cite{Serrin1963,KozonoSohr1996,FarwigSohrVarnhorn2012}),
justified here by approximation in the space $V$, gives
\begin{equation}
\label{eq:weak-strong}
    \frac12\|w(t)\|_2^2
    +\int_0^t\|\nabla w(s)\|_2^2\,\dd s
    \leq
    \int_0^t\|\nabla v(s)\|_{L^\infty(D)}\|w(s)\|_2^2\,\dd s
\end{equation}
for $0\leq t\leq T'$.  The boundary terms vanish exactly as above.  Gronwall's
lemma implies $u=v$ on $[0,T']$.  Since $T'<T$ is arbitrary, $u=v$ almost
everywhere on $[0,T)$.  Weak continuity in $H$, together with
\eqref{eq:strong-trace}, then also gives $u(T)=v(T)$.  This proves uniqueness.
\end{proof}

\begin{proof}[Proof of \cref{thm:main}]
Condition \eqref{eq:pq-condition} is equivalent to
\[
    3-\frac2p-\frac2q<1.
\]
Choose $\alpha$ as in \eqref{eq:alpha-intro} and perform the construction in
\cref{sec:profile,sec:construction}.  The preceding construction gives a classical solution with the required boundary conditions.  Applying
\cref{lem:force-rate} with $(\sigma,m)=(q,p)$ gives $f\in L^q_tL^p_x$, while applying
it with $(\sigma,m)=(1,2)$ gives $f\in L^1_tL^2_x$.  The rates and the blow-up follow
from \cref{lem:force-rate,lem:velocity-rates}, while
\cref{lem:velocity-mixed} gives \eqref{eq:velocity-mixed-intro}--
\eqref{eq:velocity-endpoints-intro}.  Finally,
\cref{prop:extension} gives the Leray--Hopf extension, the energy equality and
uniqueness.
\end{proof}

\begin{proof}[Proof of \cref{cor:curve}]
The first two assertions follow from $1/p+1/q>1$.  For the last choose
\[
 \max\left\{0,3-\frac2p-\frac2q,
 1-\frac2m-\frac2\sigma\right\}<\alpha<1
\]
and apply \eqref{eq:velocity-mixed-intro}.
\end{proof}

\section{Conclusion}
The construction gives the parallel exact conditions
\[
 f\in L^q_tL^p_x\iff\alpha>3-\frac2p-\frac2q,
 \qquad
 v\in L^\sigma_tL^m_x\iff\alpha>1-\frac2m-\frac2\sigma,
\]
with the endpoint modifications above.  Thus energy control, strong mixed-norm
integrability, uniqueness and the energy equality do not prevent terminal loss
of boundedness under an $L^1_tL^2_x$ force.

\section*{Acknowledgements}
Hugo Beir\~ao da Veiga is partially supported by FCT (Portugal) under project
UIDB/MAT/04561/2020.  Jiaqi Yang is supported by the National Natural Science
Foundation of China under Grant No.~12471225, and Natural Science Basic Research Program of Shaanxi (Program No. 2026JC-YXQN-02).

\section*{Data availability and conflict of interest}
No datasets were generated or analysed in this work.  The authors declare that they
have no conflict of interest.

\end{document}